\documentclass[11pt, oneside]{article}   	
\usepackage{geometry}                		
\usepackage{amsmath, amsthm}     
\usepackage{amsfonts} 
\usepackage{amssymb}
\usepackage{tikz}
\usetikzlibrary{arrows}
\usepackage{graphicx}				
\newtheorem{thm}{Theorem}[section]
\newtheorem{cor}[thm]{Corollary}
\newtheorem{lem}[thm]{Lemma}
\newtheorem{fig}[thm]{Figure}

\theoremstyle{remark}
\newtheorem{rem}[thm]{Remark}

\theoremstyle{definition}
\newtheorem{nota}[thm]{Notation}

\theoremstyle{definition}
\newtheorem{defi}[thm]{Definition}

\theoremstyle{definition}
\newtheorem{ex}[thm]{Example}

\title{Baumslag-Solitar graphs through colours}
\author{Katie Buchhorn}

\begin{document}
\maketitle
\begin{center}
Supervised by Sam Webster and Nathan Brownlowe
\end{center}
\subsection*{Abstract}
Higher-rank graphs are, as the name suggests, higher-dimensional analogues of directed graphs which we will define using category theory. The whole idea of my project was to construct what we call a Baumslag-Solitar graph, a higher-rank graph. They're too abstract an idea to picture, but in building Baumslag-Solitar graphs through coloured graphs, we are able to do so. I modelled my work on constructing Baumslag-Solitar graphs off Hazlewood, Raeburn,  Sims, and Webster's \emph{Remarks on some fundamental results about higher-rank graphs and their $C^*$-algebras} [2]. This is a novel approach in understanding the generalisation of directed graphs.

\section{Motivation in constructing $k$-graphs}
\subsection{The basic idea}
\begin{defi}
A \emph{directed graph} $E = ( E^0, E^1, r, s)$ consists of two countable sets $E^0, E^1$ and functions $r,s : E^1 \rightarrow E^0$.
\end{defi}
\vspace{7pt}

The set $E^0$ is the set of vertices, $E^1$ is the set of edges, $r$ is called the range map and $s$ is called the source map.
\vspace{7pt}

For $f\in E^1$, we have
\[
\begin{tikzpicture}[>=stealth,xscale=3]
\node[inner sep = 0.8pt, label=below:{$r(f)$}] (00) at (0,0) {$\scriptstyle\bullet$};
\node[inner sep = 0.8pt, label=below:{$s(f)$}] (10) at (1,0) {$\scriptstyle\bullet$};
\draw[black, <-] (00) to node[auto,black]{$f$} (10) ;
\end{tikzpicture}
\]

A \emph{path of length $n$} is a sequence of edges $\mu=\mu_1\mu_2...\mu_n$ such that $s(\mu_i)=r(\mu_{i+1})$. Denote by $E^n$ the set of all paths of length $n$ and define $E^*:= \bigcup_{n\in \mathbb{N}} E^n$.
\vspace{7pt}

We will take this idea of a directed graph and generalise it to a higher-dimension, then an abstract semigroup. This generalised graph will be defined using categorical framework.

\begin{defi}\label{category}
A \emph{category} $\mathcal{C}$ consists of two sets $C^0$ and $C^*$, two maps $r,s: C^* \rightarrow C^0$, a partially defined product $(f,g) \mapsto fg$ from $\{ (f,g) \in C^* \times C^* : s(f) = r(g)\}$ to $C^*$, and distinguished elements $\{ \iota_v \in C^* : v \in C^0\}$ satisfying
\begin{itemize}
\item[(i)]$r(fg) = r(f)$ and $s(fg) = s(g)$;
\item[(ii)]$(fg)h = f(gh)$ when $s(f) = r(g)$ and $s(g) = r(h)$; 
\item[(iii)] $r(\iota_v) = v = s(\iota_v)$ and $\iota_v f = f$, $g\iota_v = g$ when $r(f) = v$ and $s(g) = v$.
\end{itemize}
\end{defi}
The similarity between directed graphs and categories is the reason why the framework of categories is the natural choice in defining the generalisation of directed graphs.

\subsection{Generalisation of a directed graph in $\mathbb{N}^k$}
Instead of the length of a path having values in $\mathbb{N}$, we will explore the higher-dimensional analogue of a directed graph, first with the length of a path defined in $\mathbb{N}^k$ as was first done by Hazelwood, Raeburn, Sims and Webster [2], and secondly with the length of a path defined in the Baumslag-Solitar semigroup.
\vspace{10pt}

A \emph{$k$-coloured graph} is a directed graph $E$ together with a map $c:E^1\rightarrow \{c_1,c_2,...,c_k\}$. So a path $x \in E^*$ has both colour and shape determined by $q: \{c_1,c_2,...,c_k\}\rightarrow \mathbb{N}^k$ where $q(c_i)=e_i$ for all $i$. For $x\in E^*$, define $d^*(x)= q(c(x))$. We call $d^*(x)$ the \emph{degree} of a path $x\in E^*$ which we can see is assigned some higher-dimensional value in $\mathbb{N}^k$.

\begin{defi}
For $m \in (\mathbb{N} \cup \{ \infty \} )^k$, define a $k$-coloured graph $E_{k,m}$ by
\begin{align*}
&E^0_{k,m} = \{n \in \mathbb{N}^k: 0 \le n \le m\} \\
&E^1_{k,m} = \{n + v_i: n, n + e_i \in E^0_{k,m} \} \\
&r(n+v_i) = n, s(n+v_i) = n + e_i, \text {and } c(n + v_i) = c_i 
\end{align*}
\end{defi}

A \emph{graph morphism} $\psi$ from a directed graph $E$ to a directed graph $F$  consists of functions $\psi^0:E^0 \rightarrow F^0$ and $\psi^1:E^1 \rightarrow F^1$ such that $r_F(\psi^1(e))= \psi^0(r_E(e))$ and $s_F(\psi^1(e))=\psi^0(s_E(e))$ for all $e\in E^1$. A \emph{coloured-graph morphism} is a graph morphism $\psi$ such that $c_E(e) = c_F(\psi(e))$ for every $e \in E^1$.
\vspace{10pt}

For a given coloured-graph morphism $\lambda: E_{k,m} \rightarrow E$ we say $\lambda$ has degree $m$ and write $d(\lambda)=m$, and define $r(\lambda):= \lambda(0)$ and $s(\lambda) := \lambda(m)$. 

\vspace{10pt}

For $p,q \in \mathbb{N}^k$ with $p \le q$, define $E_{k,[p,q]}$ to be the subgraph of $E_{k,q}$ such that 
\begin{align*}
&E_{k,[p,q]}^0 = \{ n \in \mathbb{N}^k : p \le n \le q\} \\
&E_{k,[p,q]}^1 = \{ x \in E_{k,q}^1: s(x), r(x) \in E_{k,[p,q]}^0\}
\end{align*}

Given a coloured-graph morphism $\lambda: E_{k,m} \rightarrow E$ and $p,q \in \mathbb{N}^k$ such that $p \le q \le m$, define $\lambda|_{E_{k,[p,q]}}^*: E_{k,[p,q]} \rightarrow E$ by
\begin{align*}
\lambda|_{E_{k,[p,q]}}^*(a) = \lambda(p+a)
\end{align*}
The star is to remind us that this non-standard restriction involves a translation.
\vspace{7pt}

Given any $k$-coloured graph $E$ and distinct $i,j \in \{1, ..., k\}$, an \emph{$\{i, j\}$- square} in $E$ is a coloured directed graph morphism $\lambda : E_{k, e_i + e_j} \rightarrow E$. A collection $\mathcal{C}$ of squares in $E$  is \emph{complete} if each $x \in F^*$ with $c(x) = c_ic_j$ and $i \neq j$, there is a unique $\phi \in \mathcal{C}$ such that $x = \phi(v_i)\phi(e_i + v_j)$. We write $\phi(v_i)\phi(e_i + v_j) \sim \phi(v_j)\phi(e_j + v_i)$. 
\vspace{7pt}

If $\lambda:E_{k,m}\rightarrow E$ is a coloured-graph morphism and $\phi$ is a square in $E$, then $\phi$ \emph{occurs in} $\lambda$ if there exists $n\in \mathbb{N}^k$ such that $\phi(x)=\lambda(x+n)$ for all $x\in E_{k, e_i + e_j}$.
\vspace{7pt}

Let $E$ be a $k$-coloured graph, and $m\in \mathbb{N}^k \setminus \{0\}$. Fix $x\in E^*$ and a coloured-graph morphism $\lambda : E_{k,m} \rightarrow E$. We say $x$ \emph{traverses} $\lambda$ if the shape $q(c(x))$ of $x$ is equal to $d^*(x)$ and $\lambda(q(c(x_1x_2...x_{l-1}))+v_{c_{x_l}})=x_l$ for all $0<l\le |m|$. If $m=0$, then $x\in E^0$ and $dom(\lambda)=\{0\}$, and we say $x$ traverses $\lambda$ if $x=\lambda(0)$.

\subsection{Some major results}
\begin{thm}\label{trav}Let $E$ be a 2-coloured graph and let $\mathcal{C}$ be a complete collection of squares in $E$. For every $x \in E^*$ there is a unique $\mathcal{C}$-compatible coloured-graph morphism $\lambda_x:E_{2, d(x)}\rightarrow E$ such that $x$ traverses $\lambda_x$.
\end{thm}

\begin{proof}
We prove by induction over $|x|$. If $|x|=0$, then $x \in E^0$. Define the coloured-graph morphism $\lambda_x: E_{2, (0,0)} \rightarrow E$ by $\lambda_x((0,0)) \mapsto x$. Then $d(\lambda_x)=q(c(x))=(0,0)$ and $\lambda_x((0,0))=x$. Hence $x$ traverses $\lambda_x$.

\vspace{10pt}

$\lambda_x$ maps a single point  $(0,0) \in E_{2, (0,0)} $ to $x \in E$, there is only one way this can be done hence $\lambda_x$ is unique. No squares occur in $\lambda_x$, so it is trivial to see that $\lambda_x$ is $\mathcal{C}$-compatible.

\vspace{10pt}

Now suppose, for the inductive hypothesis, that for each $y\in E^*$ with $y\le n$ there exists a unique coloured-graph morphism $\lambda_y: E_{2, d(y)} \rightarrow E$ such that $y$ traverses $\lambda_y$. Fix a path $x\in E^*$ with $|x|=n+1$, and express $x=yf$ where $f\in E^1$ with $c(f)=c_i$. Let $m:= q(c(y))$. By the inductive hypothesis, there exists a unique $\mathcal{C}$-compatible coloured-graph morphism $\lambda_y$ such that $y$ traverses $\lambda_y$.

\vspace{10pt}

Consider the set $\{j\le 2: j \not= i, m_j>0\}$ as colours in the path $y$ that differ from $c(f)=c_i$ where $m=(m_1,m_2)$. Note that $m_j$ can be thought of as expressing how many edges of colour $j$ exist in $y$. Since we are only dealing with a maximum of two colours, we have two cases: (1) edges in $y$ are all colour $i$, or, (2) at least one edge in $y$ is colour $j$. That is,
\begin{align*}
|\{ j\le 2: j \not=i, m_j>0 \}| = \begin{cases} 0 \text{ :Case 1} \\
1 \text{ :Case 2} \end{cases}
\end{align*}

For Case 1, we can see that  $E_{2,d(x)} = E_{2,m} \cup E_{[m, m+e_i]}$.

\begin{fig}{Image of $E_{2,d(x)} $ for $i=1$}
\[
\begin{tikzpicture}[>=stealth,xscale=3]

\node[inner sep = 0.8pt, label=below:{$(0,0)$}] (00) at (0,0) {$\scriptstyle\bullet$};
\node[inner sep = 0.8pt, label=below:{$(1,0)$}] (10) at (1,0) {$\scriptstyle\bullet$};
\node[inner sep = 0.8pt, label=below:{$(2,0)$}] (20) at (2,0) {$\scriptstyle\bullet$};
\node[inner sep = 0.8pt] (30) at (3,0) {...};
\node[inner sep = 0.8pt, label=below:{$m$}] (40) at (3.1,0) {$\scriptstyle\bullet$};
\node[inner sep = 0.8pt, label=below:{$m +e_1$}] (50) at (4.1,0) {$\scriptstyle\bullet$};

\draw[blue, <-] (00) to node[auto,black]{$(0,0)+v_1$} (10) ;
\draw[blue, <-] (10) to node[auto,black]{$(1,0)+v_1$} (20) ;
\draw[blue, <-] (20) to node[auto,black]{$(2,0)+v_1$} (30) ;
\draw[blue, <-] (40) to node[auto,black]{$m+v_1$} (50) ;

\end{tikzpicture}
\]
\end{fig}

So define the coloured graph morphism $\lambda_x$ by the formulae
\begin{align}
\lambda_x |_{E_{2,m}} = \lambda_y,  \lambda_x(m+v_i) = f \text{ and } \lambda_x(m+e_i)=s(f)
\end{align}

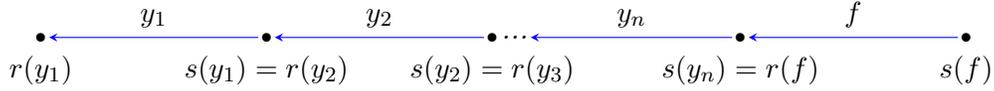
\begin{fig}{Image of $\lambda_x$ for $i=1$}
\[
\begin{tikzpicture}[>=stealth,xscale=3]
\node[inner sep = 0.8pt, label=below:{$r(y_1)$}] (00) at (0,0) {$\scriptstyle\bullet$};
\node[inner sep = 0.8pt, label=below:{$s(y_1)=r(y_2)$}] (10) at (1,0) {$\scriptstyle\bullet$};
\node[inner sep = 0.8pt, label=below:{$s(y_2)= r(y_3)$}] (20) at (2,0) {$\scriptstyle\bullet$};
\node[inner sep = 0.8pt] (30) at (2.1,0) {...};
\node[inner sep = 0.8pt, label=below:{$s(y_n)=r(f)$}] (40) at (3.1,0) {$\scriptstyle\bullet$};
\node[inner sep = 0.8pt, label=below:{$s(f)$}] (50) at (4.1,0) {$\scriptstyle\bullet$};

\draw[blue, <-] (00) to node[auto,black]{$y_1$} (10) ;
\draw[blue, <-] (10) to node[auto,black]{$y_2$} (20) ;
\draw[blue, <-] (30) to node[auto,black]{$y_n$} (40) ;
\draw[blue, <-] (40) to node[auto,black]{$f$} (50) ;

\end{tikzpicture}
\]
\end{fig}

Hence $x$ traverses $\lambda_x$. To see that $\lambda_x$ is unique, suppose $x$ also traverses a coloured-graph morphism $\mu$. Then $d(x)=q(c(x))=d(\mu)$ and $\mu(q(c(x_1x_2...x_{l-1})) + v_{c(x_l)})= x_l$ for all $0 < l \le | n+1|$. Which roughly says that $\mu$ maps all successive edges in $E_{2,d(x)} $ to the path $x=yf \in E^*$, so then $\mu$ satisfies the formula (1) hence $\mu = \lambda_x$.

\vspace{10pt}

For Case 2 we have $i \not=j$ satisfying $m_j > 0$. Then
\begin{align*}
E_{2,m+e_i} = E_{2,m} \cup (E_{2, m+e_i -e_j}) \cup (E_{2, [m-e_j, m+e_i] })
\end{align*}

Claim 1: Let $\phi^j$ be the unique square in $\mathcal{C}$ traversed by $\lambda_y((m-e_j) + v_j)f$. Let $g^j = \phi^j(0+v_i)$ and $h^j = \phi^j(e_i + v_j)$, so $g^jh^j \sim \lambda_y((m-e_j)+v_j)f$. Then there is a unique coloured-graph morphism $\lambda^j: E_{2, m-e_j+e_i} \rightarrow E$ such that $\lambda^j|_{E_{2, m-e_j}}= \lambda_y|_{E_{2, m-e_j}}$ and $\lambda^j((m-e_j)+v_i)=g^j$.

\vspace{10pt}

Fix a path $z^j$ which traverses $\lambda_y|_{E_{2,m-e_j}}$. Observe that $|z^jg^j|=n$, so the inductive hypothesis gives a unique $\mathcal{C}$-compatible coloured-graph morphism $\lambda^j: E_{2, m-e_j+e_i} \rightarrow E$ such that $z^jg^j$ traverses $\lambda^j$. So then $z^j$ traverses $\lambda^j|_{E_{2,m-e_j}}$. By the uniqueness of the $\mathcal{C}$-compatible coloured-graph morphism $\lambda^j$, we have $\lambda^j|_{E_{2, m-e_j}}= \lambda_y|_{E_{2, m-e_j}}$ proving Claim 1.

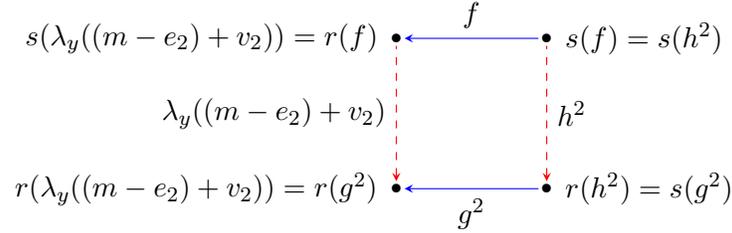
\vspace{10pt}

\begin{fig}{Image of $E_{2, m+e_i}$ for $i=1$}
\[
\begin{tikzpicture}[>=stealth,xscale=2, yscale=2]

\node[inner sep = 0.8pt, label=below:{$(0,0)$}] (00) at (0,0) {$\scriptstyle\bullet$};
\node[inner sep = 0.8pt, label=below:{$(1,0)$}] (10) at (1,0) {$\scriptstyle\bullet$};
\node[inner sep = 0.8pt, label=below:{$(0,1)$}] (01) at (0,1) {$\scriptstyle\bullet$};
\node[inner sep = 0.8pt, label=below:{$(2,0)$}] (20) at (2,0) {$\scriptstyle\bullet$};
\node[inner sep = 0.8pt, label=below:{$(0,2)$}] (02) at (0,2) {$\scriptstyle\bullet$};
\node[inner sep = 0.8pt, label=below:{$(1,1)$}] (11) at (1,1) {$\scriptstyle\bullet$};
\node[inner sep = 0.8pt] (03) at (0,3) {$\scriptstyle\bullet$};
\node[inner sep = 0.8pt] (30) at (3,0) {$\scriptstyle\bullet$};
\node[inner sep = 0.8pt] (13) at (1,3) {$\cdots$};
\node[inner sep = 0.8pt] (31) at (3,1) {$\vdots$};
\node[inner sep = 0.8pt] (12) at (1,2) {$\cdots$};
\node[inner sep = 0.8pt] (21) at (2,1) {$\vdots$};
\node[inner sep = 0.8pt] (02.5) at (0,2.5) {$\vdots$};
\node[inner sep = 0.8pt] (2.50) at (2.5,0) {$\cdots$};

\node[inner sep = 0.8pt, label=below:{$m+e_1$}] (33) at (3,3) {$\scriptstyle\bullet$};
\node[inner sep = 0.8pt, label=below:{$m$}] (23) at (2,3) {$\scriptstyle\bullet$};
\node[inner sep = 0.8pt, label=below:{$m +e_1-e_2$}] (32) at (3,2) {$\scriptstyle\bullet$};
\node[inner sep = 0.8pt, label=below:{$m -e_2$}] (22) at (2,2) {$\scriptstyle\bullet$};

\draw[blue, <-] (00) to node[auto,black]{$$} (10) ;
\draw[red, <-, dashed] (00) to node[auto,black]{$$} (01) ;
\draw[blue, <-] (10) to node[auto,black]{$$} (20) ;
\draw[red, <-, dashed] (01) to node[auto,black]{$$} (02) ;
\draw[blue, <-] (01) to node[auto,black]{$$} (11) ;
\draw[red, <-, dashed] (10) to node[auto,black]{$$} (11) ;

\draw[blue, <-] (23) to node[auto,black]{$$} (33) ;
\draw[red, <-, dashed] (22) to node[auto,black]{$$} (23) ;
\draw[blue, <-] (22) to node[auto,black]{$$} (32) ;
\draw[red, <-, dashed] (32) to node[auto,black]{$$} (33) ;

\end{tikzpicture}
\]
\end{fig}
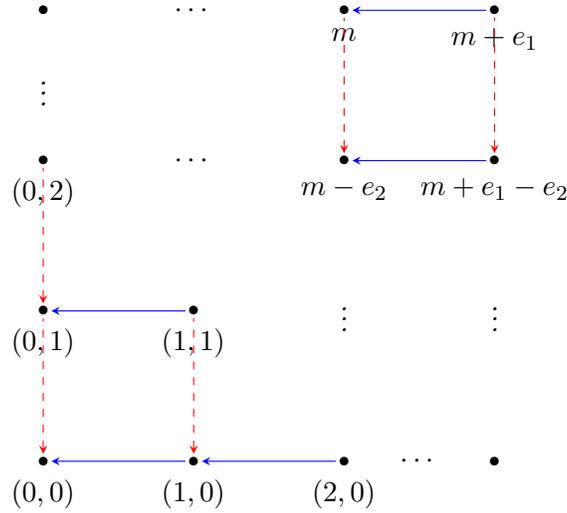
\begin{fig}{Image of $\phi^j$ for $j=2$}
\[
\begin{tikzpicture}[>=stealth,xscale=2, yscale=2]
\node[inner sep = 0.8pt, label=left:{$r(\lambda_y((m-e_2)+v_2))=r(g^2)$}] (00) at (0,0) {$\scriptstyle\bullet$};
\node[inner sep = 0.8pt, label=right:{$r(h^2)=s(g^2)$}] (10) at (1,0) {$\scriptstyle\bullet$};
\node[inner sep = 0.8pt, label=left:{$s(\lambda_y((m-e_2)+v_2))=r(f)$}] (01) at (0,1) {$\scriptstyle\bullet$};
\node[inner sep = 0.8pt, label=right:{$s(f)=s(h^2)$}] (11) at (1,1) {$\scriptstyle\bullet$};

\draw[blue, <-] (00) to node[auto,black, swap]{$g^2$} (10) ;
\draw[red, <-, dashed] (00) to node[auto,black]{$\lambda_y((m-e_2)+v_2)$} (01) ;
\draw[blue, <-] (01) to node[auto,black]{$f$} (11) ;
\draw[red, <-, dashed] (10) to node[auto,black,swap]{$h^2$} (11) ;

\end{tikzpicture}
\]
\end{fig}
\vspace{10pt}
Define the function $\lambda_x: E_{2, m+e_i} \rightarrow E$ via the formula
\begin{align*}
\lambda_x|_{E_{2,m}}=\lambda_y, \lambda_x|_{E_{2, m+e_i-e_j}}=\lambda^j \text{ and } \lambda_x|_{E_{2, [m-e_j, m+e_i] }}= \phi^j
\end{align*}
This $\lambda_x$ is well defined on its domain from Claim 1. Since $\phi^j$ is a square in $\mathcal{C}$, with $\lambda_y$ and $\lambda^j$ $\mathcal{C}$-compatible, then $\lambda_x$ is $\mathcal{C}$-compatible, and $x=yf$ traverses $\lambda_x$.

\vspace{10pt}

For uniqueness, fix a $\mathcal{C}$-compatible coloured-graph morphism $\mu$ traversed by $x$. 
We wish to show that $\mu= \lambda_x$, we will show this piece by piece.
Since $x=yf$ traverses $\mu$, $y$ traverses $\mu|_{E_{2,m}}$. The inductive hypothesis says that 
\begin{align*}
\mu|_{E_{2,m}}= \lambda_y =\lambda_x|_{E_{2,m}}.
\end{align*}
Now, $z^j$ is a traversal of $\lambda_y|_{E_{2,m-e_j}}= \lambda_x|_{E_{2,m-e_j}}=\mu|_{E_{2,m-e_j}}$. Then $z^j \lambda_y(m-e_j + v_j)$ traverses $\lambda_y$. So $z^j \lambda_y((m-e_j )+ v_j)f$ traverses $\mu$. Now $\phi^j$ is the unique square in $\mathcal{C}$ traversed by $\lambda_y((m-e_j )+ v_j)f$. So
\begin{align*}
\mu|_{2, [m-e_j, m+e_1]}^* = \phi^j = \lambda_x|_{2, [m-e_j, m+e_1]}^*
\end{align*}
Now $z^jg^j$ traverses $\mu|_{E_{2,m+e_i -e_j}}$ so the inductive hypothesis forces
\begin{align*}
\mu|_{E_{2,m+e_i -e_j}} = \lambda^j = \lambda_x|_{E_{2,m+e_i -e_j}} 
\end{align*}
Hence $\mu=\lambda_x$, completing the proof for the second case.
\end{proof}

\begin{nota}  Fix a 2-coloured graph $E$ and a complete collection of squares $\mathcal{C}$ in $E$. For each $m \in \mathbb{N}^2$, we write $\Lambda^m$ for the set of all $\mathcal{C}$-compatible coloured-graph morphisms $\lambda: E_{2,m}\rightarrow E$. For $v \in E^0$ we define $\lambda_v: E_{2,0} \rightarrow E$ by $\lambda_v(0)=v$. Let 
\begin{align*}
\Lambda = \cup_{m\in  \mathbb{N}^2} \Lambda^m. 
\end{align*}
We have previously defined $d: \Lambda \rightarrow \mathbb{N}^2$ and $r,s: \Lambda \rightarrow \Lambda^0$.
\end{nota}

\begin{thm}If $\mu: E_{2,m} \rightarrow E$ and $\nu: E_{2,n} \rightarrow E$ are $\mathcal{C}$-compatible coloured-graph morphisms such that $s(\mu)=r(\nu)$, then there exists a unique $\mathcal{C}$-compatible coloured-graph morphism $\mu\nu: E_{m+n} \rightarrow E$ such that $(\mu\nu)|_{E_{2,m}}=\mu$ and $(\mu\nu)|^*_{E_{2,[m, m+n]}}= \nu$. Under this composition map, the set $\Lambda$ is a category.
\end{thm}

\begin{proof}
Fix $x,y \in E^*$ such that $x$ traverses $\mu$ and $y$ traverses $\nu$. Lemma~{lem2} implies that $xy$ traverses a unique $\mathcal{C}$-compatible coloured-graph morphism $\mu\nu$. Then $x$ traverses $(\mu\nu)|_{E_{2,m}}$, and $y$ traverses $(\mu\nu)|_{E_{2,[m,m+n]}}$, so the uniqueness of the  coloured-graph morphism in Theorem - gives $(\mu\nu)|_{E_{2,m}}=\mu$ and $(\mu\nu)|_{E_{2,[m,m+n]}}=\nu$. 
That same uniqueness property forces any other coloured-graph morphism $\lambda$ with $\lambda|_{E_{2,m}}=\mu$ and $\lambda|_{E_{2,[m,m+n]}}=\nu$ to equal $\mu\nu$, because $xy$ traverse both $\lambda$ and $\mu\nu$.

\vspace{10pt}

It remains to show that $\Lambda$ is a category. For composable $\mu$, $\nu$ we have
\begin{align*}
s(\mu\nu)=(\mu\nu)(d(\mu\nu))=(\mu\nu)|^*_{E_{2,[d(\mu), d(\mu) + d(\nu)]}}(d(\nu))= \nu(d(\nu))=s(\nu)
\end{align*}
Similarly, $r(\mu\nu) = r(\mu)$. So condition (i) in Definition~\ref{category} holds.

\vspace{10pt}

For condition (ii) in Definition~\ref{category} of associativity, let $\lambda: E_{2,l} \rightarrow E$, $\mu: E_{2,m} \rightarrow E$ and $\nu: E_{2,n} \rightarrow E$ be $\mathcal{C}$-compatible coloured-graph morphisms such that $s(\lambda)=r(\mu)$ and $s(\mu)=r(\nu)$. Fix $x_{\lambda}$, $x_{\mu}$, $x_{\nu} \in E^*$ such that $x_{\lambda}$ traverses $\lambda$, $x_{\mu}$ traverses $\mu$ and $x_{\nu}$ traverses $\nu$.
Then $x_{\lambda}x_{\mu}$ traverses $\lambda\mu$. Hence $x_{\lambda}x_{\mu}x_{\nu}= (x_{\lambda}x_{\mu})x_{\nu}$ traverses $(\lambda\mu)\nu$. Similarly, $x_{\lambda}x_{\mu}x_{\nu}= x_{\lambda}(x_{\mu}x_{\nu})$ traverses $\lambda(\mu\nu)$. So Theorem~\ref{trav} implies that $(\lambda\mu)\nu=\lambda(\mu\nu)$

\vspace{10pt}

Now for condition (iii) in Definition~\ref{category}, for $v\in E^0$, we have $r(\lambda_{v})= \lambda_{v}(0)=v$ and $s(\lambda_{v})= \lambda_{v}(d(\lambda_{v}))= \lambda_{v}(0)= v$ by definition. Further, if $r(\mu) = v = s(\lambda_{v})$, then $\lambda_{v}\mu= \mu$ and if $s(\nu)= v = r(\lambda_{v})$, then $\nu =\lambda_v\ \nu$.
\end{proof}

\section{Generalisation of a directed graph in the Baumslag-Solitar semigroup}

\begin{defi}
For each pair of non-zero integers $m$ and $n$, there is a corresponding group
\begin{align*}
BS(m,n) = \langle a,b \mid ab^m = b^na \rangle 
\end{align*}
\end{defi}

Denote $BS(m,n)^+$ as the semigroup of all reduced elements in $BS(m,n)$ that contain no inverses $a^{-1}$ or $b^{-1}$. We will only be dealing with the semigroup $BS(2,1)^+$ for the remainder of this report, hence the semigroup relation is $ab^2=ba$.
\vspace{7pt}

Fix a 2-coloured graph $E$. We have the colour map $c: E^1\rightarrow \{c_a,c_b\}$. Let $c_a$:=\emph{red} and $c_b$:=\emph{blue}. Define $q:\{c_a,c_b\}\rightarrow BS(2,1)^+$ with $q(c_l)=l$ for $l\in \{a,b\}$. Hence the degree of a path $x \in E^*$ is now $d^*(x)= q(c(x)) \in BS(2,1)^+$.

\begin{defi}
For $w \in BS(2,1)^+$, we define the 2-coloured graph $E_{w}$ by
\begin{align*}
E_w^0 &= \{ z \in BS(2,1)^+ : z \le w\},
&E_w^1 &= \{(z, z a): z, za \in E_w^0\} \bigcup \{(z, z b): z, zb \in E_w^0\} \\
&&&= \{(z, z l): z, z l \in E_w^0\text{ and } l = a,b\} 
\end{align*}
With the maps $r(z, z l) = z$, $s(z, zl) = z l$ and $c(z, z l) = c_l$.
\end{defi}

\begin{fig}{The 2- coloured graph $E_{ba}$}
\[
\begin{tikzpicture}[>=stealth,xscale=4, yscale=4]

\node[inner sep = 0.8pt, label=below:{$e$}] (00) at (0,0) {$\scriptstyle\bullet$};
\node[inner sep = 0.8pt, label=below:{$b$}] (10) at (1,0) {$\scriptstyle\bullet$};
\node[inner sep = 0.8pt, label=above:{$a$}] (01) at (0,1) {$\scriptstyle\bullet$};
\node[inner sep = 0.8pt, label=above:{$ab$}] (051) at (0.5,1) {$\scriptstyle\bullet$};
\node[inner sep = 0.8pt, label=above:{$ab^2=ba$}] (11) at (1,1) {$\scriptstyle\bullet$};

\draw[blue, <-] (00) to node[auto,black]{$(e, eb)$} (10) ;
\draw[blue, <-] (01) to node[auto,black, swap]{$(a, ab)$} (051) ;
\draw[blue, <-] (051) to node[auto,black, swap]{$(ab, ab^2)$} (11) ;

\draw[red, <-, dashed] (00) to node[auto,black]{$(e,a)$} (01) ;
\draw[red, <-, dashed] (10) to node[auto,black, swap]{$(b, ba)$} (11) ;

\end{tikzpicture}
\]
\end{fig}

\begin{defi}\label{degree}
For a given coloured-graph morphism $\lambda: E_{w} \rightarrow E$ we say $\lambda$ has degree $w\in BS(2,1)^+$ and write $d(\lambda)=w$. Define $r(\lambda):= \lambda(e)$ and $s(\lambda):=\lambda(w)$. For $v\in E^0$ we define $\lambda_v: E_e \rightarrow E$ by $\lambda_v(e)=v$. 
\end{defi}

Let $E$ be a 2 coloured-graph. A collection of squares $\mathcal{C}$ in $E$ is \emph{complete} if:
\begin{itemize}
\item[1. ] for each $x\in E^*$ with $c(x) = c_ac_bc_b$, there exists a unique $\varphi \in \mathcal{C}$ such that $x= \varphi(e, a)\varphi(a, ab)\varphi(ab, ab^2)$; and\\
\item[2. ] for each $x\in E^*$ with $c(x)= c_bc_a$, there exists a unique $\varphi \in \mathcal{C}$ such that $x = \varphi(e, b)\varphi(b,ba)$.
\end{itemize}
\vspace{7pt}

If $\lambda: E_w \rightarrow E$ is a coloured-graph morphism and $\varphi: E_{ba} \rightarrow E$ is a square in $E$, then $\varphi$ \emph{occurs in }$\lambda$ if :
\begin{itemize}
\item[1. ]  there exists $m\in BS(2,1)^+$ such that $\lambda (zm)= \varphi (z)$ for all $z \in E^0_{ba}$, and,
\item[2. ] $\lambda (zm, z ml)= \varphi (z, z l)$ for all $(z, zl) \in E^1_{ba}$
\end{itemize}
A coloured-graph morphism $\lambda: E_w \rightarrow E$ is $\mathcal{C}$-\emph{compatible} if every square occurring in $\lambda$ belongs to the collection $\mathcal{C}$. 
\vspace{7pt}

For a path $x\in E^*$ and  coloured-graph morphism $\lambda: E_w \rightarrow E$, we say $x$ \emph{traverses} $\lambda$ if $d(\lambda)=d^*(x)$ and if $\lambda(q(c(x_1x_2...x_{n-1})), q(c(x_1x_2...x_{n-1}))q(c(x_n))) = x_n$ for all $0<n\le |x|$. 

For $w_1,w_2 \in BS(2,1)^+$ with $w_2=w_1w'$ for $w' \in BS(2,1)^+$, define $E_{[w_1,w_2]}$ to be the subgraph of $E_{w_2}$ such that 
\begin{align*}
E_{[w_1,w_2]}^0=& \{w\in BS(2,1)^+: w_1 \text{ is a subword of } w , w\text{ is a subword of }  w_2 \} \\
E_{[w_1,w_2]}^*=&\{ x\in E^1_{w_2}: s(x),r(x) \in E_{[w_1,w_2]}^0 \}
\end{align*}
Given a coloured-graph morphism $\lambda: E_w\rightarrow E$ and $w_1,w_2 \in BS(2,1)^+$ such that $w_2=w_1w'$ and $w=w_2w''$ for $w',w'' \in BS(2,1)^+$ define $\lambda|^*_{[w_1,w_2]}: E_{w_1^{-1}w_2} \rightarrow E$ by
\begin{align*}
\lambda|^*_{[w_1,w_2]}(a)= \lambda(w_1a)
\end{align*}

\begin{ex}
Let $E$ be the 2-coloured graph as in Figure~\ref{E}.
\begin{fig}\label{E}
\[
\begin{tikzpicture}[>=stealth,xscale=2, yscale=2]

\node[inner sep = 0.8pt, label=above:{$u$}] (u) at (0,0) {$\scriptstyle\bullet$}
    edge[->,blue,in=-135,out=135,looseness=25] node[auto,swap,black] {$g$} (u); 
   \node[inner sep = 0.8pt, label=above:{$v$}] (v) at (1,0) {$\scriptstyle\bullet$}
    edge[<-,red, dashed,in=-60,out=-120] node[auto,black] {$h$} (u)
    edge[->,red, dashed,in=60,out=120] node[auto,swap,black] {$f$} (u)
    edge[<-,blue,in=-45,out=45,looseness=25] node[auto,black] {$k$} (v);

\end{tikzpicture}
\]
\end{fig}

Squares in $E$ are $\varphi_1$ and $\varphi_2$ in Figure~\ref{squares}.

\begin{fig}\label{squares}
\[
\begin{tikzpicture}[>=stealth,xscale=3, yscale=3]
\node at (0.5,0.5) {$\varphi_1$};
\node at (2.5,0.5) {$\varphi_2$};
\node[inner sep = 0.8pt, label=below:{$u$}] (00) at (0,0) {$\scriptstyle\bullet$};
\node[inner sep = 0.8pt, label=below:{$u$}] (10) at (1,0) {$\scriptstyle\bullet$};
\node[inner sep = 0.8pt, label=above:{$v$}] (01) at (0,1) {$\scriptstyle\bullet$};
\node[inner sep = 0.8pt, label=above:{$v$}] (051) at (0.5,1) {$\scriptstyle\bullet$};
\node[inner sep = 0.8pt, label=above:{$v$}] (11) at (1,1) {$\scriptstyle\bullet$};
\draw[blue, <-] (00) to node[auto,black,swap]{$g$} (10) ;
\draw[blue, <-] (01) to node[pos=0.5,anchor=south,black]{$k$} (051) ;
\draw[blue, <-] (051) to node[pos=0.5,anchor=south,black]{$k$} (11) ;
\draw[red, <-, dashed] (00) to node[auto,black]{$f$} (01) ;
\draw[red, <-, dashed] (10) to node[auto,black,swap]{$f$} (11) ;

\node[inner sep = 0.8pt, label=below:{$v$}] (20) at (2,0) {$\scriptstyle\bullet$};
\node[inner sep = 0.8pt, label=below:{$v$}] (30) at (3,0) {$\scriptstyle\bullet$};
\node[inner sep = 0.8pt, label=above:{$u$}] (21) at (2,1) {$\scriptstyle\bullet$};
\node[inner sep = 0.8pt, label=above:{$u$}] (251) at (2.5,1) {$\scriptstyle\bullet$};
\node[inner sep = 0.8pt, label=above:{$u$}] (31) at (3,1) {$\scriptstyle\bullet$};
\draw[blue, <-] (20) to node[auto,black,swap]{$k$} (30) ;
\draw[blue, <-] (21) to node[pos=0.5,anchor=south,black]{$g$} (251) ;
\draw[blue, <-] (251) to node[pos=0.5,anchor=south,black]{$g$} (31) ;
\draw[red, <-, dashed] (20) to node[auto,black]{$h$} (21) ;
\draw[red, <-, dashed] (30) to node[auto,black,swap]{$h$} (31) ;
\end{tikzpicture}
\]
\end{fig}

We can see that $E$ has two blue-blue-red paths, and they are $kkf$ and $ggh$, represented only once in the collection $\mathcal{C}=\{ \varphi_1, \varphi_2\}$, as are the red-blue paths $hk$ and $fg$ in $E$. So $ \{ \varphi_1, \varphi_2\}$ form a complete collection.

Let $\lambda:E_{b^2a^2}\rightarrow E$ be defined as in Figure~\ref{morph}
\begin{fig}\label{morph}
\[
\begin{tikzpicture}[>=stealth,xscale=2, yscale=2]
\node at (1,1) {$\varphi_1$};
\node at (3,1) {$\varphi_1$};
\node at (0.5,2.5) {$\varphi_2$};
\node at (1.5,2.5) {$\varphi_2$};
\node at (2.5,2.5) {$\varphi_2$};
\node at (3.5,2.5) {$\varphi_2$};

\node[inner sep = 0.8pt, label=below:{$u$}] (00) at (0,0) {$\scriptstyle\bullet$};
\node[inner sep = 0.8pt, label=below:{$u$}] (20) at (2,0) {$\scriptstyle\bullet$};
\node[inner sep = 0.8pt, label=below:{$u$}] (40) at (4,0) {$\scriptstyle\bullet$};
\node[inner sep = 0.8pt, label=left:{$v$}] (02) at (0,2) {$\scriptstyle\bullet$};
\node[inner sep = 0.8pt, label=below:{$v$}] (12) at (1,2) {$\scriptstyle\bullet$};
\node[inner sep = 0.8pt, label=below:{$v$}] (22) at (2,2) {$\scriptstyle\bullet$};
\node[inner sep = 0.8pt, label=below:{$v$}] (32) at (3,2) {$\scriptstyle\bullet$};
\node[inner sep = 0.8pt, label=right:{$v$}] (42) at (4,2) {$\scriptstyle\bullet$};
\node[inner sep = 0.8pt, label=left:{$u$}] (03) at (0,3) {$\scriptstyle\bullet$};
\node[inner sep = 0.8pt, label=above:{$u$}] (053) at (0.5,3) {$\scriptstyle\bullet$};
\node[inner sep = 0.8pt, label=above:{$u$}] (13) at (1,3) {$\scriptstyle\bullet$};
\node[inner sep = 0.8pt, label=above:{$u$}] (153) at (1.5,3) {$\scriptstyle\bullet$};
\node[inner sep = 0.8pt, label=above:{$u$}] (23) at (2,3) {$\scriptstyle\bullet$};
\node[inner sep = 0.8pt, label=above:{$u$}] (253) at (2.5,3) {$\scriptstyle\bullet$};
\node[inner sep = 0.8pt, label=above:{$u$}] (33) at (3,3) {$\scriptstyle\bullet$};
\node[inner sep = 0.8pt, label=above:{$u$}] (353) at (3.5,3) {$\scriptstyle\bullet$};
\node[inner sep = 0.8pt, label=right:{$u$}] (43) at (4,3) {$\scriptstyle\bullet$};

\draw[blue, <-] (00) to node[auto,black, swap]{$g$} (20) ;
\draw[blue, <-] (20) to node[auto,black, swap]{$g$} (40) ;
\draw[blue, <-] (02) to node[auto,black, swap]{$k$} (12) ;
\draw[blue, <-] (12) to node[auto,black, swap]{$k$} (22) ;
\draw[blue, <-] (22) to node[auto,black, swap]{$k$} (32) ;
\draw[blue, <-] (32) to node[auto,black, swap]{$k$} (42) ;

\draw[blue, <-] (03) to node[auto,black]{$g$} (053) ;
\draw[blue, <-] (053) to node[auto,black]{$g$} (13) ;
\draw[blue, <-] (13) to node[auto,black]{$g$} (153) ;
\draw[blue, <-] (153) to node[auto,black]{$g$} (23) ;
\draw[blue, <-] (23) to node[auto,black]{$g$} (253) ;
\draw[blue, <-] (253) to node[auto,black]{$g$} (33) ;
\draw[blue, <-] (33) to node[auto,black]{$g$} (353) ;
\draw[blue, <-] (353) to node[auto,black]{$g$} (43) ;

\draw[red, <-, dashed] (00) to node[auto,black]{$f$} (02) ;
\draw[red, <-, dashed] (20) to node[auto,black]{$f$} (22) ;
\draw[red, <-, dashed] (40) to node[auto,swap,black]{$f$} (42) ;
\draw[red, <-, dashed] (02) to node[auto,black]{$h$} (03) ;
\draw[red, <-, dashed] (12) to node[auto,black]{$h$} (13) ;
\draw[red, <-, dashed] (22) to node[auto,black]{$h$} (23) ;
\draw[red, <-, dashed] (32) to node[auto,black]{$h$} (33) ;
\draw[red, <-, dashed] (42) to node[auto,black,swap]{$h$} (43) ;

\end{tikzpicture}
\]
\end{fig}

So $\varphi_1$ occurs in $\lambda$ because $\lambda^*_{[b,b^2a]}=\varphi_1$. In fact, every square occurring in $\lambda$ is either $\varphi_1$ or $\varphi_2$, so $\lambda$ is $\mathcal{C}-$compatible. Also, for example, the path $ggfh\in E^*$ traverses $\lambda$.
\end{ex}

\begin{rem} Every $w \in BS(2,1)^+$ can be expressed as
\begin{align*}
b^{m_1}a^{n_1}ba^{n_2}ba^{n_3}b...ba^{n_0}  \text{ and } a^nb^m
\end{align*}
where $n_i \not= 0$ $ \forall i\not= 0$.
\end{rem}

\begin{defi}
Let $E$ be a 2-coloured graph, and let $\lambda: E_w \rightarrow E$ be a coloured-graph morphism, with $x \in E^*$ a traversal of $\lambda$. Define $\lfloor x \rfloor \in E^*$ to be a traversal of $\lambda$ such that $|x|$ is minimised. Expressing $w=b^ma^{n_1}ba^{n_2}ba^{n_3}b...ba^{n_0}$, then the shortest path traversing $\lambda$  is of the form:
\begin{align*}
\lfloor x \rfloor =& \lambda(e,b)\lambda(b,b^2)...\lambda(b^{m-1},b^m) \text{ or }\\
\lfloor x \rfloor =& \lambda(e,a)\lambda(a,a^2)...\lambda(a^{n_1-1}, a^{n_1})\lambda(a^{n_1}, a^{n_1}b)\lambda(a^{n_1}b, a^{n_1}ba)...\lambda(a^{n_1}ba^{n_2-1}, a^{n_1}ba^{n_2})\\
&\lambda(a^{n_1}ba^{n_2-1}, a^{n_1}ba^{n_2}b)...\lambda(b^ma^{n_1}ba^{n_2}ba^{n_3}b...a^{n_k}, b^ma^{n_1}ba^{n_2}ba^{n_3}b...a^{n_k}b)  \text{ or }\\
\lfloor x \rfloor =& \lambda(e,a)\lambda(a,a^2)...\lambda(a^{n_1-1}, a^{n_1})\lambda(a^{n_1}, a^{n_1}b)\lambda(a^{n_1}b, a^{n_1}ba)...\lambda(a^{n_1}ba^{n_2-1}, a^{n_1}ba^{n_2})\\
&\lambda(a^{n_1}ba^{n_2-1}, a^{n_1}ba^{n_2}b)...\lambda(b^ma^{n_1}ba^{n_2}ba^{n_3}b...a^{n_k}ba^{n_0-1}  , b^ma^{n_1}ba^{n_2}ba^{n_3}b...a^{n_k}ba^{n_0} )
\end{align*}
Define $\lceil x \rceil \in E^*$ to be a traversal of $\lambda$ such that $|x|$ is maximised. Expressing $w=a^nb^m$, we have the longest path traversing $\lambda$ is:
\begin{align*}
\lceil x \rceil =& \lambda(e,b)\lambda(b,b^2)...\lambda(b^{m-1},b^m) \text{ or }\\
\lceil x \rceil=&\lambda(e,a)\lambda(a,a^2)...\lambda(a^{n-1},a^n) \text{ or }\\
\lceil x \rceil=&\lambda(e,a)\lambda(a,a^2)...\lambda(a^n,a^nb)...\lambda(b^{m-1}, b^m)
\end{align*}
\end{defi}

This implies that every coloured-graph morphism has a traversal.
\begin{ex}
Let $\lambda: E_{b^2a^2} \rightarrow E$ be a coloured-graph morphism. Let $x \in E^*$ traverse $\lambda$. 

Then the shortest path traversing $\lambda$ is $\lfloor x \rfloor$ with 
\begin{align*}
d^*(\lfloor x \rfloor)&= q(c(\lfloor x \rfloor)) \\
&= q(c_bc_bc_ac_a) \\
&= b^2a^2
\end{align*}

The longest path traversing $\lambda$ is $\lceil x \rceil$ with 
\begin{align*}
d^*(\lceil x \rceil)&= q(c(\lceil x \rceil)) \\
&= q(c_ac_ac_bc_bc_bc_bc_bc_bc_bc_b) \\
&= a^2b^8
\end{align*}
\end{ex}
Note that $b^2a^2 = a^2b^8 \in BS(2,1)^+$, this illustrates how paths of the same degree can differ in length.

\begin{lem}\label{lem2}
Let $E$ be a 2 coloured-graph and let $\lambda:E_w \rightarrow E$ be a coloured-graph morphism where $w\in BS(2,1)^+$. Write $w=w_1w_2$ where $w_1,w_2 \in BS(2,1)^+$. If $x\in E^*$ traverses $\lambda|_{E_{[e,w_1]}}$ and $y\in E^*$ traverses $\lambda|^*_{E_{[w_1, w_1w_2]}}$, then $xy$ traverses $\lambda$.
\end{lem}

\begin{proof}
We have $d^*(\lambda)=w = w_1w_2 = q(c(x))q(c(y))= q(c(xy))$ and for $m \le |xy|$, we have
\begin{align*}
&\lambda(d^*((xy)_1...(xy)_{m-1}), d^*((xy)_1...(xy)_{m-1})q(c(xy_m)))\\
&= \begin{cases} 
\lambda|_{E_{[e,w_1]}}(d^*((xy)_1...(xy)_{m-1}), d^*((xy)_1...(xy)_{m-1})q(c(xy_m))) &\text{ if $m \le |x|$}\\
\lambda|^*_{E_{[w_1, w_1w_2]}}(d((xy)_1...(xy)_{m-|x|-1}), d((xy)_1...(xy)_{m-|x|-1})q(c(xy_{m-|x|}))) &\text{ otherwise} \end{cases}\\
&= \begin{cases} 
x_m &\text{ if $m \le |x|$}\\
y_{m-|x|} &\text{ otherwise} \end{cases}
\end{align*}
\end{proof}

\begin{thm}\label{thm1}
Let $E$ be a 2-coloured graph and let $\mathcal{C}$ be a complete collection of squares in $E$. For every $x \in E^*$ there is a unique $\mathcal{C}$-compatible coloured-graph morphism $\lambda_x: E_{d^*(x)} \rightarrow E$.
\end{thm}

We will prove by induction over $|x|$. If $|x|=0$ so that $x$ is a vertex, then $\lambda_x: E_e \rightarrow E$ with $\lambda_x(e) \mapsto x$ is the unique coloured-graph morphism such that $x$ traverses $\lambda_x$. That $\lambda_x$ is $\mathcal{C}$-compatible is trivial given that no squares exist in $\lambda_x$.

\vspace{7pt}

Suppose, for the inductive hypothesis, that for every $y \in E^*$ with $|y| \le n$, the path $y$ traverses a unique coloured graph morphism $\lambda_y: E_{d^*(y)} \rightarrow E$. Fix a path $x \in E^*$ with $|x| = n+1$. Express $x=yf$ where $y \in E^*$ and $f \in E^1$ and $c(f)= c_l$ where $l \in \{a,b\}$.

\vspace{7pt}

Let $w:= d^*(y) = q(c(y)) $. Express $w$ in its shortest form $w= b^ma^{n_1}ba^{n_2}ba^{n_3}b...a^{n_k}ba^{n_0}$ where $n_i \not= 0$ $ \forall i\not= 0$. Then either

\begin{itemize}
\item[Case 1:] $n_0 =0$ so that $ w  = b^ma^{n_1}ba^{n_2}ba^{n_3}b...a^{n_k}b$ where $n_i \not= 0$ $ \forall i \not= 0$
\item[Case 2:] $n_0 \not =0$ so that $w  = b^ma^{n_1}ba^{n_2}ba^{n_3}b...a^{n_k}ba^{n_0}$ where $n_i \not= 0$ $ \forall i$
\end{itemize}

For Case 1: $E_{w}$ contains, by definition, all $w'\in BS(2,1)^+$ such that $w=w'w''$ for some $w''\in BS(2,1)^+$, and edges whose source and range belong to the set of vertices and differ by a basis element, characterising $E_{w}$ by the blue edge $(wb^{-1}, w)$ seen in Figure~\ref{fig1}.

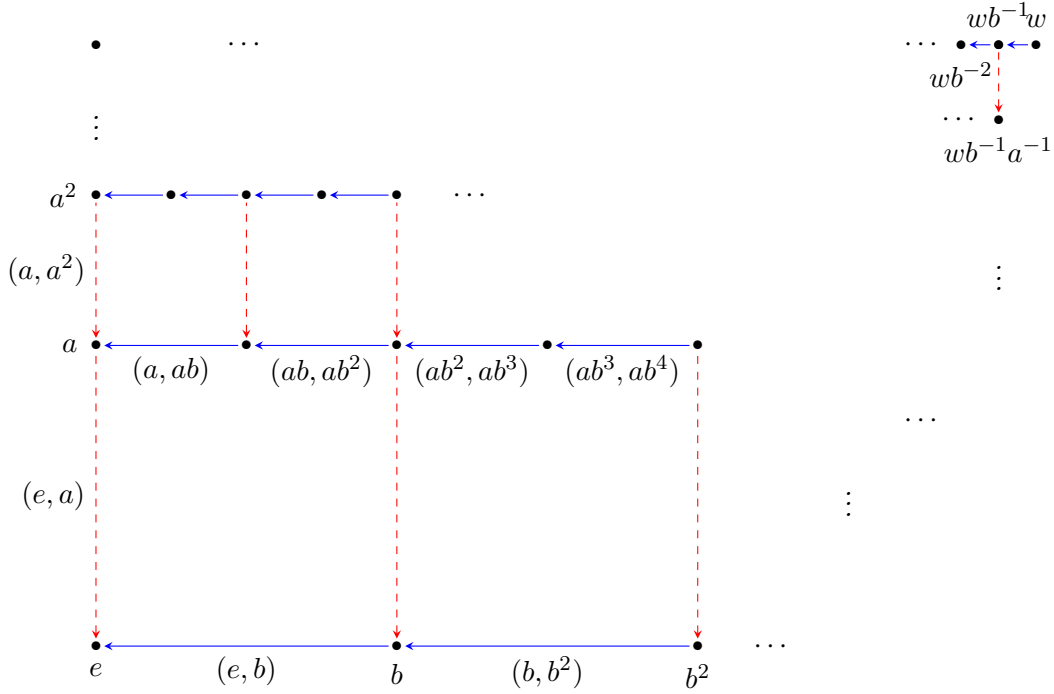
\begin{fig}{Image of $E_w$ as the domain of $\lambda_y$}\label{fig1}
\[
\begin{tikzpicture}[>=stealth,xscale=2, yscale=2]
\node[inner sep = 0.8pt, label=below:{$e$}] (00) at (0,0) {$\scriptstyle\bullet$};
\node[inner sep = 0.8pt, label=below:{$b$}] (20) at (2,0) {$\scriptstyle\bullet$};
\node[inner sep = 0.8pt, label=below:{$b^2$}] (40) at (4,0) {$\scriptstyle\bullet$};
\node[inner sep = 0.8pt, label=below:{$$}] (450) at (4.5,0) {$\cdots$};
\node[inner sep = 0.8pt, label=below:{$$}] (253) at (2.5,3) {$\cdots$};
\node[inner sep = 0.8pt, label=below:{$$}] (14) at (1,4) {$\cdots$};
\node[inner sep = 0.8pt, label=below:{$$}] (5515) at (5.5,1.5) {$\cdots$};
\node[inner sep = 0.8pt, label=below:{$$}] (625) at (6,2.5) {$\vdots$};
\node[inner sep = 0.8pt, label=below:{$$}] (51) at (5,1) {$\vdots$};
\node[inner sep = 0.8pt, label=below:{$$}] (035) at (0,3.5) {$\vdots$};
\node[inner sep = 0.8pt, label=left:{$a$}] (02) at (0,2) {$\scriptstyle\bullet$};
\node[inner sep = 0.8pt, label=below:{$$}] (12) at (1,2) {$\scriptstyle\bullet$};
\node[inner sep = 0.8pt, label=below:{$$}] (22) at (2,2) {$\scriptstyle\bullet$};
\node[inner sep = 0.8pt, label=below:{$$}] (32) at (3,2) {$\scriptstyle\bullet$};
\node[inner sep = 0.8pt, label=below:{$$}] (42) at (4,2) {$\scriptstyle\bullet$};
\node[inner sep = 0.8pt, label=left:{$a^2$}] (03) at (0,3) {$\scriptstyle\bullet$};
\node[inner sep = 0.8pt, label=below:{$$}] (053) at (0.5,3) {$\scriptstyle\bullet$};
\node[inner sep = 0.8pt, label=below:{$$}] (13) at (1,3) {$\scriptstyle\bullet$};
\node[inner sep = 0.8pt, label=below:{$$}] (153) at (1.5,3) {$\scriptstyle\bullet$};
\node[inner sep = 0.8pt, label=below:{$$}] (23) at (2,3) {$\scriptstyle\bullet$};
\node[inner sep = 0.8pt, label=below:{$$}] (04) at (0,4) {$\scriptstyle\bullet$};
\node[inner sep = 0.8pt, label=below:{$$}] (64) at (6,4) {$\scriptstyle\bullet$};
\node[inner sep = 0.8pt, label=below:{$wb^{-2}$}] (5754) at (5.75,4) {$\scriptstyle\bullet$};
\node[inner sep = 0.8pt, label=above:{$wb^{-1}$}] (64) at (6,4) {$\scriptstyle\bullet$};
\node[inner sep = 0.8pt, label=above:{$w$}] (6254) at (6.25,4) {$\scriptstyle\bullet$};
\node[inner sep = 0.8pt, label=below:{$wb^{-1}a^{-1}$}] (635) at (6,3.5) {$\scriptstyle\bullet$};
\node[inner sep = 0.8pt, label=below:{$$}] (554) at (5.5,4) {$\cdots$};
\node[inner sep = 0.8pt, label=below:{$$}] (5753) at (5.75,3.5) {$\cdots$};

\draw[blue, <-] (00) to node[auto,black, swap]{$(e,b)$} (20) ;
\draw[blue, <-] (20) to node[auto,black, swap]{$(b,b^2)$} (40) ;
\draw[blue, <-] (02) to node[auto,black, swap]{$(a,ab)$} (12) ;
\draw[blue, <-] (12) to node[auto,black, swap]{$(ab,ab^2)$} (22) ;
\draw[blue, <-] (22) to node[auto,black, swap]{$(ab^2,ab^3)$} (32) ;
\draw[blue, <-] (32) to node[auto,black, swap]{$(ab^3,ab^4)$} (42) ;

\draw[blue, <-] (03) to node[auto,black]{$$} (053) ;
\draw[blue, <-] (053) to node[auto,black]{$$} (13) ;
\draw[blue, <-] (13) to node[auto,black]{$$} (153) ;
\draw[blue, <-] (153) to node[auto,black]{$$} (23) ;
\draw[blue, <-] (5754) to node[auto,black]{$$} (64) ;
\draw[blue, <-] (64) to node[auto,black]{$$} (6254) ;

\draw[red, <-, dashed] (00) to node[auto,black]{$(e,a)$} (02) ;
\draw[red, <-, dashed] (20) to node[auto,black]{$$} (22) ;
\draw[red, <-, dashed] (40) to node[auto,black]{$$} (42) ;
\draw[red, <-, dashed] (02) to node[auto,black]{$(a,a^2)$} (03) ;
\draw[red, <-, dashed] (12) to node[auto,black]{$$} (13) ;
\draw[red, <-, dashed] (22) to node[auto,black]{$$} (23) ;
\draw[red, <-, dashed] (635) to node[auto,black]{$$} (64) ;

\end{tikzpicture}
\]
\end{fig}

Suppose first that $c(f)=c_b$.  We have $|y|= n$, so $| \lfloor y \rfloor| \le n$. We know now that that path $\lfloor y \rfloor$ ends in a red then a blue edge. Write $\lfloor y \rfloor =  \lfloor y \rfloor_1 \lfloor y \rfloor_2... \lfloor y \rfloor_{k}$ and define $y' = \lfloor y \rfloor_1 \lfloor y \rfloor_2... \lfloor y \rfloor_{k-2}$ so that $|y'| \le n-2$. Hence  $y'$  traverses $\lambda_y | _{E_{wb^{-1}a^{-1}}}$. 

Let $\phi$ be the unique square in $\mathcal{C}$ traversed by $\lambda(wb^{-1}a^{-1}, wb^{-1})\lambda(wb^{-1},w)f$ so
\begin{align*}
\lambda(wb^{-1}a^{-1}, wb^{-1})\lambda(wb^{-1},w)f \sim \phi(wb^{-1}a^{-1}, wba^{-1})\phi(wb,wba^{-1})
\end{align*}
\begin{fig}{Image of $\phi$}
\[
\begin{tikzpicture}[>=stealth,xscale=4, yscale=4]

\node[inner sep = 0.8pt, label=below:{$$}] (00) at (0,0) {$\scriptstyle\bullet$};
\node[inner sep = 0.8pt, label=below:{$$}] (10) at (1,0) {$\scriptstyle\bullet$};
\node[inner sep = 0.8pt, label=above:{$$}] (01) at (0,1) {$\scriptstyle\bullet$};
\node[inner sep = 0.8pt, label=above:{$$}] (051) at (0.5,1) {$\scriptstyle\bullet$};
\node[inner sep = 0.8pt, label=above:{$$}] (11) at (1,1) {$\scriptstyle\bullet$};

\draw[blue, <-] (00) to node[auto,black,swap]{$\phi( wb^{-1}a^{-1}, wba^{-1})$} (10) ;
\draw[blue, <-] (01) to node[auto,black]{$\lambda(wb^{-1},w)$} (051) ;
\draw[blue, <-] (051) to node[auto,black]{$f$} (11) ;

\draw[red, <-, dashed] (00) to node[auto,black]{$\lambda(wb^{-1}a^{-1}, wb^{-1})$} (01) ;
\draw[red, <-, dashed] (10) to node[auto,black, swap]{$\phi(wb,wba^{-1})$} (11) ;

\end{tikzpicture}
\]
\end{fig}

Let $x'= y'\phi(wb^{-1}a^{-1}, wba^{-1})\phi(wb,wba^{-1})$. We have $|x'| \le n$ with $d^*(x')= wb$, so the inductive hypothesis implies that $x'$ traverses a unique $\mathcal{C}$-compatible coloured-graph morphism $\lambda_{x'} :E_{wb} \rightarrow E$.

\begin{lem}\label{lem1}
Suppose $z\in E^n$ has $c(z_{n-1}z_n)=c_bc_a$ and $z$ traverses a $\mathcal{C}$-compatible 2-coloured graph-morphism $\lambda:E_{d^*(z)}\rightarrow E$. Then
\begin{align*}
z_1...z_{n-2}\lambda(d^*(z)a^{-1}b^{-1}, d^*(z)a^{-1}b^{-1}a) \lambda(d^*(z)a^{-1}b^{-1}a, d^*(z)a^{-1}b^{-1}ab)\lambda(d^*(z)a^{-1}b^{-1}ab, d^*(z)a^{-1}b^{-1}ab^2)
\end{align*}
also traverses $\lambda$.
\end{lem}

\begin{proof}
Suppose $z=z_1z_2...z_n$ traverses $\lambda$ as hypothesised above. Then $d(\lambda)=q(c(z))$ and 
\begin{align*}
\lambda(q(c(z_1z_2...z_{m-1})), q(c(z_1z_2...z_{m-1}))l_m) = z_m
\end{align*} for all $0 < m \le |z|$.

Let
\begin{align*}
z_{n-1}'&=\lambda(d(z)a^{-1}b^{-1}, d(z)a^{-1}b^{-1}a) \\
z_n'&=\lambda(d(z)a^{-1}b^{-1}a, d(z)a^{-1}b^{-1}ab) \\
z_{n+1}'&=\lambda(d(z)a^{-1}b^{-1}ab, d(z)a^{-1}b^{-1}ab^2)\\
\text{ and }z'&=z_1...z_{n-2}z_{n-1}'z_n'z_{n+1}'
\end{align*}
Then 
\begin{align*}
q(c(z'))=& q(c(z_1...z_{n-2})c(z_{n-1}'z_n'z_{n+1}') \\
=&  q(c(z_1...z_{n-2})c_ac_bc_b) \\
=& q(c(z_1...z_{n-2}))ab^2 \\
=& q(c(z_1...z_{n-2}))ba \\
=& q(c(z))
\end{align*}

Because $z$ traverses $\lambda$, we know
\begin{align*}
&\lambda(q(c(z_1z_2...z_{m-1})), q(c(z_1z_2...z_{m-1}))l_m) = z_m \text{ for all $0 < m \le |z|-2$.}\\
&\lambda(q(c(z_1z_2...z_{n-2})), q(c(z_1z_2...z_{n-2}))a) = z_{n-1}' \\
&\lambda(q(c(z_1z_2...z_{n-2}z_{n-1}' )), q(c(z_1z_2...z_{n-2}z_{n-1}' ))b)= z_n' \\
&\lambda(q(c(z_1z_2...z_{n-2}z_{n-1}'z_n' )), q(c(z_1z_2...z_{n-2}z_{n-1}'z_n' ))b)= z_{n+1}'
\end{align*}

Hence $z'$ traverses $\lambda$.
\end{proof}

\emph{proof. of Theorem~\ref{thm1} continued:} Because $|x'| \le n$ and $c(x')= c_{\lfloor y \rfloor_1}c_{\lfloor y \rfloor_2}...c_bc_a$, Lemma~\ref{lem1} give the path $\lfloor y \rfloor f$ as a traversal of $\lambda'$. Since  $\lfloor y \rfloor \sim y$, then $\lfloor y \rfloor f \sim yf$. Hence $x =yf$ traverses the unique $\mathcal{C}$-compatible graph morphism $\lambda_{x'}$.

Suppose $c(f)=c_a$. Write $y=y_1 y_2.... y_k = y' y_k$ so that $|y'|= n-1$. Let $\phi'$ be the unique square in $\mathcal{C}$ traversed by $y_k f$. Let $f' = \phi'(wb^{-1}, wb^{-1}a)$.
\begin{fig}{Image of $\phi'$}
\[
\begin{tikzpicture}[>=stealth,xscale=4, yscale=4]

\node[inner sep = 0.8pt, label=below:{$$}] (00) at (0,0) {$\scriptstyle\bullet$};
\node[inner sep = 0.8pt, label=below:{$$}] (10) at (1,0) {$\scriptstyle\bullet$};
\node[inner sep = 0.8pt, label=above:{$$}] (01) at (0,1) {$\scriptstyle\bullet$};
\node[inner sep = 0.8pt, label=above:{$$}] (051) at (0.5,1) {$\scriptstyle\bullet$};
\node[inner sep = 0.8pt, label=above:{$$}] (11) at (1,1) {$\scriptstyle\bullet$};

\draw[blue, <-] (00) to node[auto,black,swap]{$y_k$} (10) ;
\draw[blue, <-] (01) to node[pos=0.1,anchor=south,black]{$\phi'(wab^{-2}, wab^{-1})$} (051) ;
\draw[blue, <-] (051) to node[pos=0.9,anchor=south,black]{$\phi'(wab^{-1}, wa)$} (11) ;

\draw[red, <-, dashed] (00) to node[auto,black]{$f'$} (01) ;
\draw[red, <-, dashed] (10) to node[auto,black, swap]{$f$} (11) ;

\end{tikzpicture}
\]
\end{fig}

We have $|y'f'| = n$, so the inductive hypothesis gives a unique $\mathcal{C}$-compatible coloured-graph morphism $\lambda_{y'f'}: E_{d^*(y'f')}\rightarrow E$ with $y'f'$ traversing $\lambda_{y'f'}$. Note that
\begin{align*}
E_{d^*(x)} = E_{d^*(y)} \cup E_{d^*(y'f')} \cup E_{[wb^{-1}, wa]}.
\end{align*}

So we can define the coloured-graph morphism $\lambda_x$ by the equations
\begin{align*}
&\lambda_x|_{E_{d^*(y'f')}}=\lambda_{y'f'} \\
&\lambda_x|_{E_{d^*(y)}} = \lambda_y \\
&\lambda_x|_{E_{[wb^{-1}, wa]}}= \phi'
\end{align*}
such that $x=yf$ traverses $\lambda_x$. Since $\lambda_{y'f'}$ and $\lambda_y$ are both $\mathcal{C}$-compatible, and $\phi'$ belongs to $\mathcal{C}$,  $\lambda_x$ is also $\mathcal{C}$-compatible. To see that $\lambda_x$ is unique, fix a $\mathcal{C}$-compatible coloured-graph morphism $\mu$ traversed by $x$. Then $y$ traverses $\mu|_{E_{d(y)}}$, so the inductive hypothesis forces $\mu|_{E_{d(y)}}= \lambda_y =\lambda_x|_{E_{d(y)}}$. That $\mu$ is $\mathcal{C}$-compatible forces $\mu|_{E_{[wb^{-1}, wa]}}= \phi'$, so that $y'f'$ traverses $\mu|_{E_{d(y'f')}}=\lambda_{y'f'}=\lambda_x|_{E_{d(y'f')}} $. This completes the proof for Case 1.

\vspace{20pt}

For Case 2:

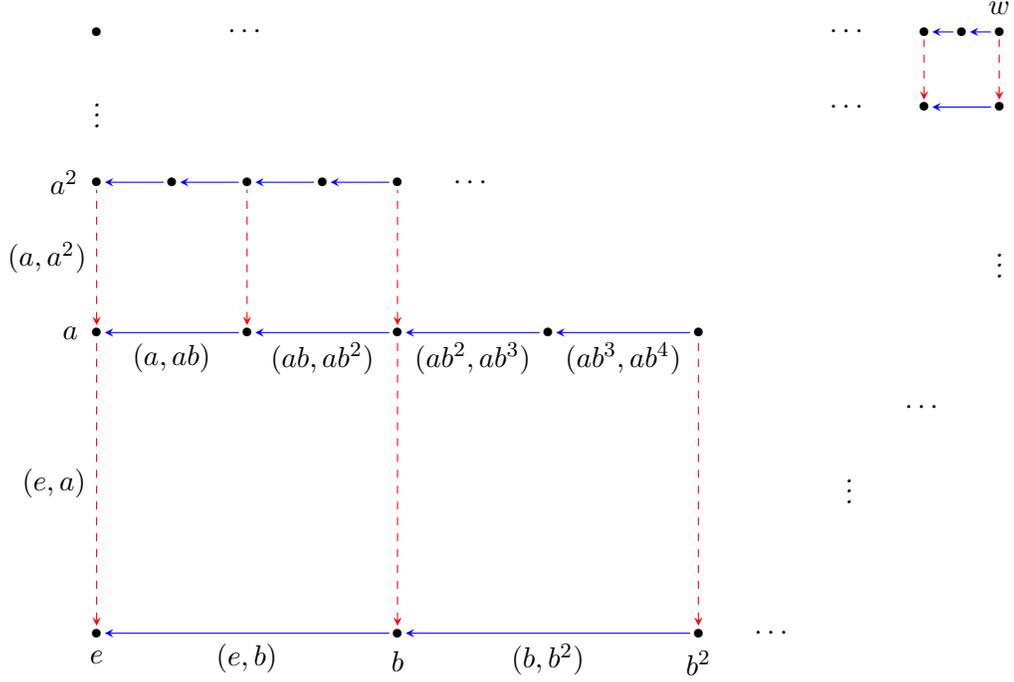
\begin{fig}{Image of the 2-coloured graph $E_w$, and domain of the coloured-graph morphism $\lambda_y$}

\[
\begin{tikzpicture}[>=stealth,xscale=2, yscale=2]
\node[inner sep = 0.8pt, label=below:{$e$}] (00) at (0,0) {$\scriptstyle\bullet$};
\node[inner sep = 0.8pt, label=below:{$b$}] (20) at (2,0) {$\scriptstyle\bullet$};
\node[inner sep = 0.8pt, label=below:{$b^2$}] (40) at (4,0) {$\scriptstyle\bullet$};
\node[inner sep = 0.8pt, label=below:{$$}] (450) at (4.5,0) {$\cdots$};
\node[inner sep = 0.8pt, label=below:{$$}] (253) at (2.5,3) {$\cdots$};
\node[inner sep = 0.8pt, label=below:{$$}] (14) at (1,4) {$\cdots$};
\node[inner sep = 0.8pt, label=below:{$$}] (5515) at (5.5,1.5) {$\cdots$};
\node[inner sep = 0.8pt, label=below:{$$}] (625) at (6,2.5) {$\vdots$};
\node[inner sep = 0.8pt, label=below:{$$}] (51) at (5,1) {$\vdots$};
\node[inner sep = 0.8pt, label=below:{$$}] (035) at (0,3.5) {$\vdots$};
\node[inner sep = 0.8pt, label=left:{$a$}] (02) at (0,2) {$\scriptstyle\bullet$};
\node[inner sep = 0.8pt, label=below:{$$}] (12) at (1,2) {$\scriptstyle\bullet$};
\node[inner sep = 0.8pt, label=below:{$$}] (22) at (2,2) {$\scriptstyle\bullet$};
\node[inner sep = 0.8pt, label=below:{$$}] (32) at (3,2) {$\scriptstyle\bullet$};
\node[inner sep = 0.8pt, label=below:{$$}] (42) at (4,2) {$\scriptstyle\bullet$};
\node[inner sep = 0.8pt, label=left:{$a^2$}] (03) at (0,3) {$\scriptstyle\bullet$};
\node[inner sep = 0.8pt, label=below:{$$}] (053) at (0.5,3) {$\scriptstyle\bullet$};
\node[inner sep = 0.8pt, label=below:{$$}] (13) at (1,3) {$\scriptstyle\bullet$};
\node[inner sep = 0.8pt, label=below:{$$}] (153) at (1.5,3) {$\scriptstyle\bullet$};
\node[inner sep = 0.8pt, label=below:{$$}] (23) at (2,3) {$\scriptstyle\bullet$};
\node[inner sep = 0.8pt, label=below:{$$}] (04) at (0,4) {$\scriptstyle\bullet$};

\node[inner sep = 0.8pt, label=above:{$$}] (554) at (5.5,4) {$\scriptstyle\bullet$};
\node[inner sep = 0.8pt, label=below:{$$}] (5754) at (5.75,4) {$\scriptstyle\bullet$};
\node[inner sep = 0.8pt, label=above:{$w$}] (64) at (6,4) {$\scriptstyle\bullet$};
\node[inner sep = 0.8pt, label=below:{$$}] (635) at (6,3.5) {$\scriptstyle\bullet$};
\node[inner sep = 0.8pt, label=above:{$$}] (5535) at (5.5,3.5) {$\scriptstyle\bullet$};

\node[inner sep = 0.8pt, label=below:{$$}] (54) at (5,4) {$\cdots$};
\node[inner sep = 0.8pt, label=below:{$$}] (535) at (5,3.5) {$\cdots$};

\draw[blue, <-] (00) to node[auto,black, swap]{$(e,b)$} (20) ;
\draw[blue, <-] (20) to node[auto,black, swap]{$(b,b^2)$} (40) ;
\draw[blue, <-] (02) to node[auto,black, swap]{$(a,ab)$} (12) ;
\draw[blue, <-] (12) to node[auto,black, swap]{$(ab,ab^2)$} (22) ;
\draw[blue, <-] (22) to node[auto,black, swap]{$(ab^2,ab^3)$} (32) ;
\draw[blue, <-] (32) to node[auto,black, swap]{$(ab^3,ab^4)$} (42) ;

\draw[blue, <-] (03) to node[auto,black]{$$} (053) ;
\draw[blue, <-] (053) to node[auto,black]{$$} (13) ;
\draw[blue, <-] (13) to node[auto,black]{$$} (153) ;
\draw[blue, <-] (153) to node[auto,black]{$$} (23) ;
\draw[blue, <-] (5754) to node[auto,black]{$$} (64) ;
\draw[blue, <-] (554) to node[auto,black]{$$} (5754) ;
\draw[blue, <-] (5535) to node[auto,black]{$$} (635) ;

\draw[red, <-, dashed] (00) to node[auto,black]{$(e,a)$} (02) ;
\draw[red, <-, dashed] (20) to node[auto,black]{$$} (22) ;
\draw[red, <-, dashed] (40) to node[auto,black]{$$} (42) ;
\draw[red, <-, dashed] (02) to node[auto,black]{$(a,a^2)$} (03) ;
\draw[red, <-, dashed] (12) to node[auto,black]{$$} (13) ;
\draw[red, <-, dashed] (22) to node[auto,black]{$$} (23) ;
\draw[red, <-, dashed] (635) to node[auto,black]{$$} (64) ;
\draw[red, <-, dashed] (5535) to node[auto,black]{$$} (554) ;

\end{tikzpicture}
\]
\end{fig}

Suppose $c(f)=c_b$. We have $E_{wb} = E_w \cup E_{[w, wb]}$. Define the coloured-graph morphism $\lambda_x$ by the formulae
\begin{align*}
\lambda_x|_{E_w} =\lambda_y|_{E_w}, \text{    }  \lambda_x(w, wb)=f   \text{ and   }  \lambda_x(wb)=s(f)
\end{align*}

By construction, $x$ traverses $\lambda_x$.  Given $\lambda_y$ is $\mathcal{C}$-compatible and unique, that $\lambda_x$ is $\mathcal{C}$-compatible and unique is trivial. 

\vspace{7pt}
Now suppose that $c(f)=c_a$. Write $d(y)=w$ in its longest form, so that $w=a^Nb^M$ for some $N,M\in  \mathbb{N}$.

Let $\phi_1$ be the unique square in $\mathcal{C}$ traversed by the blue-red path $\lambda_y(wb^{-1},w)f$.
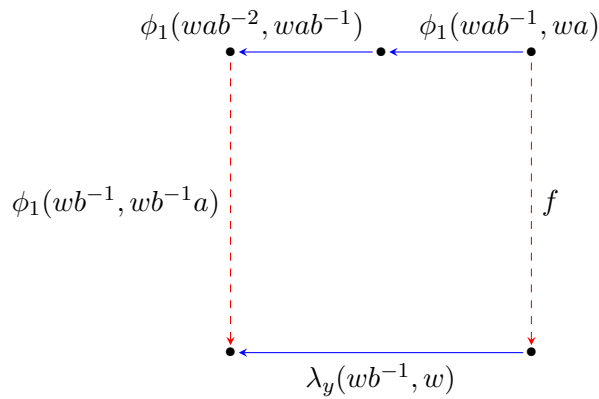
\begin{fig}{Image of $\phi_1$}
\[
\begin{tikzpicture}[>=stealth,xscale=4, yscale=4]

\node[inner sep = 0.8pt, label=below:{$$}] (00) at (0,0) {$\scriptstyle\bullet$};
\node[inner sep = 0.8pt, label=below:{$$}] (10) at (1,0) {$\scriptstyle\bullet$};
\node[inner sep = 0.8pt, label=above:{$$}] (01) at (0,1) {$\scriptstyle\bullet$};
\node[inner sep = 0.8pt, label=above:{$$}] (051) at (0.5,1) {$\scriptstyle\bullet$};
\node[inner sep = 0.8pt, label=above:{$$}] (11) at (1,1) {$\scriptstyle\bullet$};

\draw[blue, <-] (00) to node[auto,black,swap]{$\lambda_y(wb^{-1},w)$} (10) ;
\draw[blue, <-] (01) to node[pos=0.1,anchor=south,black]{$\phi_1(wab^{-2}, wab^{-1})$} (051) ;
\draw[blue, <-] (051) to node[pos=0.9,anchor=south,black]{$\phi_1(wab^{-1}, wa)$} (11) ;

\draw[red, <-, dashed] (00) to node[auto,black]{$\phi_1(wb^{-1}, wb^{-1}a)$} (01) ;
\draw[red, <-, dashed] (10) to node[auto,black,swap]{$f$} (11) ;

\end{tikzpicture}
\]
\end{fig}

Define recursively the squares  $\phi_n \in \mathcal{C}$  where $\phi_n$ is uniquely determined by the traversal $\lambda_y(wb^{-n}, wb^{-n+1})\phi_{n-1}(wb^{-n+1}, wb^{-n+1}a)$ for $1<n \le M$.

\begin{fig}{Image of $\phi_n$}
\[
\begin{tikzpicture}[>=stealth,xscale=4, yscale=4]

\node[inner sep = 0.8pt, label=below:{$$}] (00) at (0,0) {$\scriptstyle\bullet$};
\node[inner sep = 0.8pt, label=below:{$$}] (10) at (1,0) {$\scriptstyle\bullet$};
\node[inner sep = 0.8pt, label=above:{$$}] (01) at (0,1) {$\scriptstyle\bullet$};
\node[inner sep = 0.8pt, label=above:{$$}] (051) at (0.5,1) {$\scriptstyle\bullet$};
\node[inner sep = 0.8pt, label=above:{$$}] (11) at (1,1) {$\scriptstyle\bullet$};

\draw[blue, <-] (00) to node[auto,black,swap]{$\lambda_y(wb^{-n}, wb^{-n+1})$} (10) ;
\draw[blue, <-] (01) to node[pos=0.01,anchor=south,black]{$\phi_n(wb^{-n}a, wb^{-n}ab)$} (051) ;
\draw[blue, <-] (051) to node[pos=0.99,anchor=south,black]{$\phi_n(wb^{-n}ab, wb^{-n+1}a)$} (11) ;

\draw[red, <-, dashed] (00) to node[auto,black]{$\phi_n(wb^{-n},wb^{-n}a)$} (01) ;
\draw[red, <-, dashed] (10) to node[auto,black,swap]{$\phi_{n-1}(wb^{-n+1}, wb^{-n+1}a)$} (11) ;

\end{tikzpicture}
\]
\end{fig}

Define the coloured-graph morphism $\lambda_x$ by the formulae
\begin{align*}
&\lambda_x|_{E_{d^*(y)}} = \lambda_y \\
&\lambda_x|_{E_{[wb^{-n}, wb^{-n+1}a]}} = \phi_n \text{  for $1\le n \le M$}
\end{align*}

Then $x$ traverses $\lambda_x$, and $\lambda_x$ is $\mathcal{C}$-compatible because $\lambda_y$ is $\mathcal{C}$-compatible and $\phi_n \in \mathcal{C}$ for $1\le n \le M$.

\vspace{7pt}
To see that $\lambda_x$ is unique, suppose there exists a $\mathcal{C}$-compatible coloured-graph morphism $\mu$ such that $x$ traverses $\mu$. Then $y$ traverses $\mu|_{E_{d^*(y)}}$, so the inductive hypothesis forces $\lambda_x|_{E_{d^*(y)}}=\mu|_{E_{d^*(y)}} = \lambda_y$. Now we have 
\begin{align*}
\lambda_y(wb^{-n}, wb^{-n+1})= \lambda_x(wb^{-n}, wb^{-n+1}) = \mu(wb^{-n}, wb^{-n+1})\text{ for all $1\le n \le M$.}
\end{align*}
Hence $\mu|_{E_{[wb^{-n}, wb^{-n+1}a]}} = \phi_n$ for all $1\le n \le M$. So $\mu = \lambda_x$.

\begin{cor}\label{cor1}
Let $E$ be a 2-coloured graph and let $\mathcal{C}$ be a complete collection of squares in $E$. If $\mu: E_{w_1} \rightarrow E$ and $\nu: E_{w_2} \rightarrow E$ are  $\mathcal{C}$-compatible coloured-graph morphisms such that $s(\mu)=r(\nu)$, then there exists a unique $\mathcal{C}$-compatible coloured-graph morphism $\mu\nu: E_{w_1w_2} \rightarrow E$, called the composition of $\mu$ and $\nu$ such that $(\mu\nu)|_{E_{w_1}}= \mu$ and $(\mu\nu)|^*_{E_{[w_1,w_1w_2]}}=\nu$.
\end{cor}
\begin{proof}
Fix $x,y \in E^*$ such that $x$ traverses $\mu$ and $y$ traverses $\nu$. Then $s(x)=r(y)$ and Theorem~\ref{thm1} implies that $xy$ traverses a unique $\mathcal{C}$-compatible coloured-graph morphism $\mu\nu$. Restricting the domain gives that $x$ traverses $(\mu\nu)|_{E_{w_1}}$, because $d^*(x)=w_1$, and $y$ traverses $(\mu\nu)|^*_{E_{[w_1,w_1w_2]}}$, because $d^*(y)=w_2$. So the uniqueness of $(\mu\nu)|_{E_{w_1}}$ and $(\mu\nu)|^*_{E_{[w_1,w_1w_2]}}$ by Theorem~\ref{thm1} give $(\mu\nu)|_{E_{w_1}}= \mu$ and $(\mu\nu)|^*_{E_{[w_1,w_1w_2]}}=\nu$.

\vspace{7pt}
So we have a unique $\mathcal{C}$-compatible coloured-graph morphism $\mu\nu$ given traversals $x$ of $\mu$ and $y$ of $\nu$. We wish to show that $\mu\nu$ is unique given $\mu$ and $\nu$. Suppose that we start with any other coloured-graph morphism, say $\lambda$, such that $\lambda|_{E_{w_1}}=\mu$ and $\lambda|^*_{E_{[w_1,w_1w_2]}}=\nu$. So $x$ traverses $\lambda|_{E_{w_1}}$ and $y$ traverses $\lambda|^*_{E_{[w_1,w_1w_2]}}$. Then Lemma~\ref{lem2} shows that $xy$ traverses $\lambda$. Since $xy$ traverses both $\lambda$ and $\mu\nu$, the uniqueness in Theorem~\ref{thm1} forces $\lambda=\mu\nu$. Hence the composition of $\mu$ and $\nu$ is unique.
\end{proof}

\begin{rem}\label{rem1} Corollary~\ref{cor1} implies that $\mu:= \lambda|_{E_{w_1}}$ and $\nu:= \lambda|^*_{E_{[w_1,w_1w_2]}}$ satisfy $\lambda=\mu\nu$. We wish to show that $\mu$ and $\nu$ are unique.  Suppose that $\mu':E_{w_1}\rightarrow E$ and $\nu':E_{w_2} \rightarrow E$ are two other $\mathcal{C}$-compatible coloured-graph morphisms such that $\mu'\nu'=\lambda$. Now, $\mu$ and $\mu'$ are coloured-graph morphisms which both have the domain $E_{w_1}$ forcing $\mu'=\lambda|_{E_{w_1}}=\mu$. Similarly, $\nu'=\lambda|^*_{E_{[w_1,w_1w_2]}}=\nu$. So $\mu$ and $\nu$ are unique coloured-graph morphisms such that $d(\mu)=w_1$ and $d(\nu)=w_2$, and $\lambda=\mu\nu$.
\end{rem}

\subsection{Defining a $BS(2,1)^+$-graph}
In this section we will be constructing a Higher-rank graph, called the $BS(2,1)^+$-graph, given a set of coloured-graph morphisms.

\begin{defi}
Let $\mathcal{C}=(\mathcal{C}^0,\mathcal{C}^*, r_{\mathcal{C}},s_{\mathcal{C}})$ and $\mathcal{D}=(\mathcal{D}^0,\mathcal{D}^*, r_{\mathcal{D}},s_{\mathcal{D}})$ be categories. Then $F=(F^0: \mathcal{C}^0 \rightarrow \mathcal{D}^0, F^*:\mathcal{C}^* \rightarrow \mathcal{D}^*)$ is a \emph{functor} if:
\begin{itemize}
\item[(i)]$s_{\mathcal{D}}(F^*(f))=F^0(s_{\mathcal{C}}(f))$ $\forall f \in \mathcal{C}^*$
\item[(ii)]$r_{\mathcal{C}}(F^*(f))=F^0(r_{\mathcal{D}}(f))$ $\forall f \in \mathcal{C}^*$
\item[(iii)]$F^*(f g) = F^*(f) F^*(g)$ $\forall f,g \in \mathcal{C}^*$
\item[(iv)]$F^*(\iota_v)=\iota_{F^0(v)}$ $\forall v \in \mathcal{C}^0$
\end{itemize}
\end{defi}
So a functor between categories maps morphisms in $\mathcal{C}$ to morphisms in $\mathcal{D}$,  maps objects in $\mathcal{C}$ to objects in $\mathcal{D}$, and preserves the domain and codomain maps, composition, and respects the identity morphisms.

Such an $F$ is also called a \emph{covariant functor}. If we replace (iii) by $F^*(f g) = F^*(g)  F^*(f)$ $\forall f,g \in \mathcal{C}^*$ then we call $F$ a \emph{contravariant functor}.

\begin{defi}\label{Lambda}
Let $E$ be a 2-coloured graph and $\mathcal{C}$ be a complete collection of squares in $E$.
For each $w \in BS(2,1)^+$, write $\Lambda^w$ for the set of all $\mathcal{C}$-compatible coloured-graph morphisms $\lambda: E_w \rightarrow E$. Let $\Lambda=\cup_{w\in BS(2,1)^+} \Lambda^w$.
\end{defi}

\begin{defi}
A \emph{$BS(2,1)^+$-graph} is a pair $(\Lambda_{BS}, d)$, where $\Lambda_{BS}$ is a countable category and $d$ is a functor $d: \Lambda_{BS} \rightarrow BS(m,n)^+$ satisfying the \emph{Factorisation property:} For every $\lambda \in Mor(\Lambda_{BS})$ and $w_1,w_2 \in BS(2,1)^+$ with $d(\lambda)=w_1w_2$, there are unique elements $\mu, \nu \in Mor(\Lambda_{BS})$ such that $\lambda=\mu\nu$, $d(\mu)=w_1$ and $d(\nu)=w_2$.
\end{defi}

\begin{thm}\label{major}[Major Theorem]
Fix a 2-coloured graph $E$ and a complete collection of squares $\mathcal{C}$ in $E$. Define $\Lambda$ as in Definition~\ref{Lambda}, endowed with the structure maps $r,s: \Lambda \rightarrow \Lambda^0$ and with composition as defined in Corollary~\ref{cor1}. Let $d: \Lambda \rightarrow BS(2,1)^+$ be as defined in Definition~\ref{degree}, then $(\Lambda, d)$ is a $BS(2,1)^+$-graph.
\end{thm}

\subsection{Constructing $BS(2,1)^+$-graphs}
\begin{thm}\label{category}
Fix a 2-coloured graph $E$ and a complete collection of squares $\mathcal{C}$ in $E$. The set $\Lambda$, endowed with the structure maps $r,s: \Lambda \rightarrow \Lambda^0$ and with composition as defined in Corollary~\ref{cor1}, is a category.
\end{thm}
\begin{proof}
Fix $\mathcal{C}$-compatible coloured-graph morphisms $\lambda:E_{w_1}\rightarrow E$, $\mu:E_{w_2}\rightarrow E$, and $\nu:E_{w_3}\rightarrow E$ with $s(\lambda)=r(\mu)$ and $s(\mu)=r(\nu)$. Then the composition of $\mu$ and $\nu$ is a unique $\mathcal{C}$-compatible coloured-graph morphism $\mu\nu: E_{w_2w_3} \rightarrow E$,such that $(\mu\nu)|_{E_{w_2}}= \mu$ and $(\mu\nu)|^*_{E_{[w_2,w_2w_3]}}=\nu$. We have
\begin{align*}
&s(\mu\nu)=\mu\nu(w_2w_3)= \mu\nu|^*_{E_[w_2, w_2w_3]}(w_3)=\nu(w_3)=s(\nu), \text{ and }\\
&r(\mu\nu) = \mu\nu(e) = \mu\nu|_{E_{w_2}}(e) = \mu(e) = r(\mu)
\end{align*}so condition (i) holds.
\vspace{7pt}

For condition (ii), fix $x_{\lambda}, x_{\mu}, x_{\nu} \in E^*$ such that $x_{\lambda}$ traverses $\lambda$, $x_{\mu}$ traverses $\mu$ and $x_{\nu}$ traverses $\nu$. Corollary~\ref{cor1} gives the composition of $\lambda$ and $\mu$ as a unique $\mathcal{C}$-compatible coloured-graph morphism $\lambda\mu$ with $(\lambda\mu)|_{E_{w_1}}= \lambda$ and $(\lambda\mu)|^*_{E_{w_1,w_1w_2}}=\mu$. Then Lemma~\ref{lem2} gives that $x_{\lambda}x_{\mu}$ traverses $\lambda\mu$. A repeated application of Corollary~\ref{cor1} and Lemma~\ref{lem2} on traversals $x_{\lambda}x_{\mu}$ of $\lambda\mu$ and $x_{\nu}$ of $\nu$ gives
\begin{align*}
x_{\lambda} x_{\mu}x_{\nu} = (x_{\lambda} x_{\mu})x_{\nu} \text{ traverses } (\lambda\mu)\nu
\end{align*}
Similarly, $x_{\mu}x_{\nu}$ traverses $\mu\nu$, and
\begin{align*}
x_{\lambda} x_{\mu}x_{\nu} = x_{\lambda} (x_{\mu}x_{\nu}) \text{ traverses } \lambda(\mu\nu)
\end{align*}
By Theorem~\ref{thm1}, we conclude that $(\lambda\mu)\nu= \lambda(\mu\nu)$.
\end{proof}
\vspace{7pt}

Finally, for $v\in E^0$, we have previously defined the identity morphism $\lambda_v:E_e \rightarrow E$ by $\lambda_v(e) = v$. We have $r(\lambda_v)= \lambda_v(e) = v$ and $s(\lambda_v)= \lambda_v(d(\lambda_v))=\lambda_v(e) = v$. Now suppose $r(\mu)=\mu(e)=:v$, we wish to show that $\mu=\lambda_v\mu$. We have $\mu|_{E_e}: E_e \rightarrow E$ with $\mu|_{E_e}(e)= \mu(e)=v=\lambda_v(e)$. So $\lambda_v=\mu|_{E_e}$. Also, $\mu|^*_{E_{[e,w_2]}}:E_{[e, w_2]}\rightarrow E$. By definition, $E_{[e, w_2]}= E_{w_2}$ and $\mu|^*_{E_{[e,w_2]}}(a)=\mu(ea)=\mu(a)$. So then $\mu=\mu|^*_{E_{[e,w_2]}}$. Remark~\ref{rem1} gives $\mu= \mu|_{E_e}\mu|^*_{E_{[e,w_2]}}$, hence $\mu=\lambda_v\mu$.

For $s(\nu)=\nu(d(\nu))=\nu(w_3):=v$, we wish to show $\nu\lambda_v=\nu$. Restricting the domain of $\nu$ gives both $\nu|_{E_{w_3}}: E_{w_3}\rightarrow E$ and $\nu|^*_{E_{[w_3,w_3]}}: E_{[w_3,w_3]}\rightarrow E$. It is trivial to see that $\nu|_{E_{w_3}}=\nu$, and that the subgraph $E_{[w_3,w_3]}$ of $E_{w_3}$ is $E_e$. Then $\nu|^*_{E_{[w_3,w_3]}}(e)=\nu(w_3e)=\nu(w_3) = v$. So $\nu|^*_{E_{[w_3,w_3]}}=\lambda_v$. Again, Remark~\ref{rem1} gives $\nu=\nu|_{E_{w_3}}\nu|^*_{E_{[w_3,w_3]}}$. Thus $\nu\lambda_v=\nu$. Condition (iii) holds hence $\Lambda$ is a category.

\begin{lem}\label{BScat}
$BS(2,1)^+$ is a category with $Obj(BS(2,1)^+) = \{ \star \}$ and $\iota_{\star} = e$, $Mor(BS(2,1)^+) = BS(2,1)^+$, maps $r,s: w \mapsto \star$ for $w\in BS(2,1)^+$ and composition $w_1 \cdot w_1 := w_1w_2$.
\end{lem}
\begin{proof}
Fix morphisms $w_1,w_2,w_3 \in BS(2,1)^+$. There is only one object in $BS(2,1)^+$, so the source and range of any morphism in $BS(2,1)^+$ is precisely the object $\star$. It follows that
\begin{align*}
&s(w_1w_2)= \star = s(w_2)\\
&r(w_1w_2)= \star = r(w_1) \text{ $\forall w_1,w_2\in BS(2,1)^+$}
\end{align*}
So condition (i) holds. For condition (ii), we know that $(w_1w_2)w_3=w_1(w_2w_3)$ $\forall w_1,w_2,w_3 \in BS(2,1)^+$ given that associativity is inherited from the property that makes $BS(2,1)^+$ a semigroup. 
We also have $r(\iota_{\star})=r(e)=\star=s(e)=s(\iota_{\star})$ and $\iota_{\star}w_1=ew_1=w_1$ and $w_1\iota_{\star}=w_1e=w_1$ $\forall w_1 \in BS(2,1)^+$, proving condition (iii) for all $w_1 \in BS(2,1)^+$. Hence, $BS(2,1)^+$ is a category.
\end{proof}

\begin{proof}{\emph{of Theorem~\ref{major}}}

We have that $\Lambda=(\Lambda^0, \Lambda^*, r_{\Lambda}, s_{\Lambda})$ is a category, directly from Theorem~\ref{category} and that $BS(2,1)^+=(\{\star\}, BS(2,1)^+, r_{BS}, s_{BS})$ is a category from Lemma~\ref{BScat}. 

\vspace{7pt}
All that remains to show is that $d: \Lambda \rightarrow BS(2,1)^+$ is a functor which satisfies the factorisation property. We have
\begin{align*}
&d^0: \Lambda^0 \rightarrow \star \\
&d: \Lambda^* \rightarrow BS(2,1)^+
\end{align*}
Since $BS(2,1)^+$ as a category has only one object $\star$, $d$ respects $r_{BS}$ and $s_{BS}$
\begin{align*}
&s_{BS}(d(\lambda))= \star = d^0(s_{\Lambda}) \text{, and,} \\
&r_{BS}(d(\lambda))= \star = d^0(r_{\Lambda})
\end{align*} for all $\lambda \in \Lambda^*$.

To see that $d$ respects composition, fix traversals $x_{\lambda_1}$ of $\lambda_1$ and $x_{\lambda_2}$ of $\lambda_2$. Then
\begin{align*}
d(\lambda_1\lambda_2)&= d^*(x_{\lambda_1}x_{\lambda_2})\\
&=d^*(x_{\lambda_1})d^*(x_{\lambda_2}) \text{ by Corollary~\ref{cor1}}\\
&=d(\lambda_1)d^*(\lambda_2)
\end{align*}

We must now show that $d(\lambda_v)=\lambda_{d^0(v)}$ for all $v \in \Lambda^0$. We have 
\begin{align*}
d(\lambda_v) = e = \iota_{\star} = \lambda_{d^0(v)} \text{ $\forall v \in \Lambda^0$}
\end{align*}
So $d:\Lambda \rightarrow BS(2,1)^+$ is a functor. That $d$ satisfies the factorisation property follows directly from Corollary~\ref{cor1}. Hence $(\Lambda, d)$ is a $BS(2,1)^+$-graph.
\end{proof}

\newpage
\section{References}
\begin{itemize}

\item[]{[1]} N. Brownlowe, A. Sims, and S.~T. Vittadello, \emph{Co-universal $C^*$-algebras associated to generalised graphs}, Israel J. Math. \textbf{193} (2013), 399--440.

\item[]{[2]}\label{2}  R. Hazlewood, I. Raeburn, A. Sims, and S.~B.~G. Webster, \emph{Remarks on some fundamental results about higher-rank graphs and their $C^*$-algebras}, Proc. Edinburgh Math. Soc. \textbf{56} (2013), 575--597.

\item[]{[3]} A. Kumjian  and D. Pask, \emph{Higher rank graph $C^*$-algebras}, New York J. Math. \textbf{6} (2000), 1--20.
\end{itemize}

\end{document}